\documentclass[12pt]{amsart}

\usepackage{extsizes}
\usepackage{amssymb, amsmath}
\usepackage{upgreek}
\usepackage{fourier}

\usepackage{geometry}
\theoremstyle{plain}
\newtheorem{theorem}{Theorem}
\newtheorem{corollary}[theorem]{Corollary}

\newtheorem{lemma}[theorem]{Lemma}

\theoremstyle{definition}
\newtheorem*{definition}{Definition}

\theoremstyle{remark}

\newcommand{\A}{\mathcal A}
\newcommand{\As}{\mathcal A_s}
\newcommand{\Ap}{\mathcal A_+}
\newcommand{\App}{\mathcal A_+^{-1}}
\newcommand{\Am}{\mathcal B}
\newcommand{\B}{\mathcal B}
\newcommand{\Bpp}{\mathcal B_+^{-1}}
\newcommand{\Ams}{\mathcal {B}_s}

\newcommand{\Ampp}{\mathcal B_+^{-1}}
\newcommand{\DtA}{\mathcal D_\tau(\A)}
\newcommand{\DtAi}{\mathcal D_\tau^{-1}(\A)} 

\newcommand{\DtAmi}{\mathcal D_{\omega}^{-1}(\mathcal B)}
\newcommand{\cDa}{D_\alpha^c}
\newcommand{\minDa}{D_\alpha^{min}}
\newcommand{\maxDa}{D_\alpha^{max}}
\newcommand{\moDa}{D_\alpha^{mo}}
\newcommand{\Daz}{D_{\alpha,z}}
\newcommand{\cQa}{Q_\alpha^c}
\newcommand{\minQa}{Q_\alpha^{min}}
\newcommand{\maxQa}{Q_\alpha^{max}}
\newcommand{\moQa}{Q_\alpha^{mo}}
\newcommand{\Qaz}{Q_{\alpha,z}}
\newcommand{\Qazv}{Q_{\alpha,z'}}
\mathchardef\hy="2D

\newcommand{\R}{\mathbb R}

\newcommand{\ler}[1]{\left( #1 \right)}

\newcommand{\fel}{{1/2}}
\newcommand{\mfel}{{-1/2}}

\begin{document}

\title[]{Quantum R\'enyi relative entropies on density spaces of $C^*$-algebras:\\their symmetries and their essential difference}

\author{Lajos Moln\'ar}
\address{University of Szeged, Interdisciplinary Excellence Centre, Bolyai Institute,
H-6720 Szeged, Aradi v\'ertan\'uk tere 1.,
Hungary, and
Budapest University of Technology and Economics,  Institute of Mathematics,
H-1521 Budapest, Hungary}
\email{molnarl@math.u-szeged.hu}
\urladdr{http://www.math.u-szeged.hu/\~{}molnarl}

\dedicatory{}

\thanks{The research was begun while the author was visiting the University of Lille and was supported by the Labex CEMPI (ANR-11-LABX-0007-01). The author is very grateful to his host Mostafa Mbekhta for the kind hospitality. 
\break
Ministry of Human Capacities, Hungary grant 20391-3/2018/FEKUSTRAT is also acknowledged and the work was supported by the National Research, Development and Innovation Office of Hungary, NKFIH, Grant No. K115383, too.}

\begin{abstract}
We extend the definitions of different types of quantum R\'enyi relative entropy from the finite dimensional setting of density matrices to density spaces of $C^*$-algebras. We show that those quantities (which trivially coincide in the classical commutative case) are essentially different on non-commutative algebras in the sense that none of them can be transformed to another one by any surjective transformation between density spaces. Besides, we determine the symmetry groups of density spaces corresponding to each of those quantum R\'enyi relative entropies and find that they are identical. Similar results concerning the Umegaki and the Belavkin-Staszewksi relative entropies are also presented.
\end{abstract}

\maketitle

\section{Introduction and statements of the results}

Relative entropies play a very important role in classical information theory. For the purposes of measuring information content, they are used to measure how well a probability distribution approximates another one.
For any two given probability distributions $p, q$ on a finite set $\mathcal X$, their R\'enyi $\alpha$-divergence ($\alpha \in ]0,\infty[, \alpha\neq 1$) is the quantity
\[
D_\alpha(p||q) = \frac{1}
{\alpha-1} \log \sum_{x\in \mathcal X} p(x)^{\alpha}q(x)^{1-\alpha}.
\]
Among all relative entropies, the parametric family $D_\alpha(.||.)$ has a distinguished role for several reasons. 
Indeed, its elements have various desirable mathematical properties:
they are non-increasing under stochastic maps, jointly convex for
$\alpha\in ]0,1[$, jointly quasi-convex for $\alpha \in ]1,\infty[$, monotone increasing as a function of the parameter $\alpha$, and
the Kullback-Leibler divergence  (i.e., the most classical relative entropy) is their limiting case as $\alpha\to 1$. In addition to that, R\'enyi relative entropies have a great operational significance as quantifiers of the trade-off between relevant quantities in many coding problems.

Quantum information theory is a very rapidly developing area of science. In view of the above classical facts, it is not a surprise that there is a quest for finding appropriate analogues of R\'enyi relative entropy in the quantum setting. Indeed, currently this is a quite hot topic in quantum information theory, a really extensive area of research. So extensive that we do not dare to select a few papers as 'the' references, instead, we only mention the recent book \cite{Toma} of Tomamichel and otherwise refer the reader to arXiv and MathSciNet where one can easily find many related materials. 

Now, a few words about the problem of finding the appropriate quantum R\'enyi relative entropy. Due to the non-commutativity of the structure of density matrices ($n\times n$ complex positive semidefinite matrices with unit trace), there are in fact many different possible ways to define  quantum R\'enyi relative entropy which would extend the classical one. The basic problem here is which one to select, which one of them is the most useful? Parallel to that is the question of how do the different non-commutative extensions relate to each other?
As for the former problem, it does not look that there would be a definite answer. Indeed, concerning all of the so far studied extensions it has turned out that some of the required nice properties are satisfied only to certain extents, or they are fully satisfied but others are not. The picture is very complicated, and the fact is that we have a variety of notions of quantum R\'enyi type relative entropy and some of them are useful for certain reasons, some of them are so for other reasons.  

The aim of this work is twofold. On the one hand, we determine the symmetries of the density spaces of quantum systems with respect to the currently considered and studied concepts of quantum R\'enyi relative entropies. We do this in the very general setting of $C^*$-algebras that has recently been introduced by Farenick etal \cite{Far16, Far17}. This question is motivated by Wigner's famous result on quantum mechanical symmetry transformations which describes the transformations on the set of pure states (rank-one densities) that preserve the quantity of transition probability. We will see that the symmetries in question, i.e., the transformations preserving the different quantum R\'enyi relative entropies, are closely related to the Jordan *-isomorphisms between the underlying algebras. These latter transformations are the most fundamental sorts of isomorphisms between $C^*$-algebras from the quantum mechanical point of view, they are just the natural isomorphisms of quantum observables (see below). On the other hand, we show that those concepts of quantum R\'enyi relative entropy not only formally but also essentially differ from each other in the sense that if one of them can be transformed to another one by a surjective transformation between the density spaces, then the underlying $C^*$-algebras are necessarily commutative (in which case those quantities trivially coincide) meaning that the systems behind are necessarily classical, non-quantum.

We fix the notation and present the basic definitions.
First, we point out that we follow the approach which, in the finite dimensional Hilbert space framework of quantum information, considers (mixed) states as density matrices, i.e., positive semidefinite matrices with unit trace. The corresponding abstract setup was introduced in the papers \cite{Far16, Far17} by Farenick etal, what we follow below. 
Namely, let $\A$ be a (unital) $C^*$-algebra. We denote by $\Ap$ the set of all positive elements of $\A$ and by $\App$ the set of all positive invertible elements of $\A$ what we call the positive definite cone of $\A$. By a trace on $\A$ we mean a positive linear functional $\tau$ on $\A$ which satisfies $\tau(AB)=\tau(BA)$ for all $A,B\in \A$. The trace $\tau$ is said to be faithful if $\tau (A)=0$, $A\in \Ap$, implies $A=0$. Fundamental examples for $C^*$-algebras having faithful traces include UHF-algebras, finite factors, irrational rotation algebras.
For such a faithful trace $\tau$ on a $C^*$-algebra $\A$, we define the $\tau$-density space of $\A$ as the set
\begin{equation*}
\DtA=\{ a\in \Ap\, :\, \tau(a)=1\}. 
\end{equation*}
In fact, in order to define the key concepts of the paper, we denote by $\DtAi$ the set of all invertible elements of the $\tau$-density space of $\A$, i.e.,
\begin{equation*}
\DtAi=\DtA\cap \App.
\end{equation*} 
For any parameter $\alpha\in ]0,\infty[, \alpha\neq 1$, we introduce the different types of quantum R\'enyi relative entropy as follows. Actually, to each of those types names of certain researchers are attributed who defined and investigated the corresponding concepts in the context of finite quantum systems, i.e., for density matrices. We begin with the conventional (or, in another terminology, standard) R\'enyi relative entropy considered by Petz what we define here as
\begin{equation}\label{E:D1}
\tau\hy\cDa(A||B)=\frac{1}{\alpha -1}\log \tau(A^\alpha B^{1-\alpha}), \quad A,B\in \DtAi,
\end{equation}
see \cite{Toma}, p. 67, and the original source \cite{Petz86}. 
Next, the minimal (or sandwiched) R\'enyi relative entropy is the quantity
\begin{equation}\label{E:D2}
\tau\hy\minDa(A||B)=\frac{1}{\alpha -1}\log \tau \ler{B^{\frac{1-\alpha}{2\alpha}}AB^{\frac{1-\alpha}{2\alpha}}}^{\alpha}, \quad A,B\in \DtAi
\end{equation}
which is originally due to M\"uller-Lennert, Dupuis,
Szehr, Fehr and Tomamichel. See \cite{Toma}, p. 58, and also \cite{MDSFT}, \cite{WWY}.
The former two sorts of quantum R\'enyi relative entropy are particular cases of the so-called $\alpha-z$-R\'enyi relative entropies 
which were introduced by Audenaert and Datta in \cite{AD}. In the present setting we define
\begin{equation}\label{E:D3}
\tau\hy\Daz(A||B)=\frac{1}{\alpha -1}\log \tau \ler{B^{\frac{1-\alpha}{2z}}A^{\frac{\alpha}{z}}B^{\frac{1-\alpha}{2z}}}^{z}, \quad A,B\in \DtAi.
\end{equation}
Here $z>0$ is any positive real number. Clearly, if $z=1$, then we get the conventional R\'enyi relative entropy, while in the case where $z=\alpha$, we obtain the minimal R\'enyi relative entropy.
After this follows the maximal R\'enyi relative entropy which is defined by
\begin{equation}\label{E:D4}
\tau\hy\maxDa(A||B)=\frac{1}{\alpha -1}\log \tau \ler{B^{\fel}(B^\mfel AB^\mfel)^\alpha B^\fel}, \quad A,B\in \DtAi
\end{equation}
and was essentially introduced by Petz and Ruskai in \cite{PetzRuskai} (in fact, in the place of the power function with exponent $\alpha$ in \eqref{E:D4}, they considered general operator convex functions). In \cite{Mat}, Matsumoto verified a certain maximality property of that quantity, this is why we call it maximal R\'enyi relative entropy (also see paragraph 4.2.3 in \cite{Toma}). Finally,
in \cite{MosOga}, Mosonyi and Ogawa introduced and studied another type of quantum R\'enyi relative entropy which, in our present setting, is defined as
\begin{equation}\label{E:D5}
\tau\hy\moDa(A||B)=\frac{1}{\alpha -1}\log \tau \ler{\exp({\alpha\log A+(1-\alpha)\log B})}, \quad A,B\in \DtAi.
\end{equation}
These are the main concepts in the paper. Observe that
in the case of commutative algebras all those quantities coincide and we will see that the converse is also true. In fact, below we will prove the much stronger statement what we have already mentioned in the abstract as well as in the first part of the introduction, which shows that the relative entropies above are essentially different.    

Before formulating our results, we remark that the above defined quantities \eqref{E:D1}-\eqref{E:D5} can trivially be extended from $\DtAi$ to the whole positive definite cone $\App$ of $\A$. In what follows we will consider maps which are kinds of invariance transformations under pairs of such numerical quantities. Clearly, the invariance property does not change if we multiply those quantities by the common scalar $(\alpha -1)$ and then omit the function $\log$ which appears in each of the above formulas. Therefore, in order to simplify our considerations a bit, for any given numbers $\alpha\in ]0,\infty[, \alpha\neq 1$ and $z>0$, we define the following numerical quantities:
\begin{gather}
\tau\hy\cQa(A||B)=\tau(A^\alpha B^{1-\alpha}), \label{E:Q1}\\ 
\tau\hy\minQa(A||B)=\tau \ler{B^{\frac{1-\alpha}{2\alpha}}AB^{\frac{1-\alpha}{2\alpha}}}^{\alpha}, \label{E:Q2}\\
\tau\hy\Qaz(A||B)=\tau \ler{B^{\frac{1-\alpha}{2z}}A^{\frac{\alpha}{z}}B^{\frac{1-\alpha}{2z}}}^{z}, \label{E:Q3}\\
\tau\hy\maxQa(A||B)=\tau \ler{B^{\fel}(B^\mfel AB^\mfel)^\alpha B^\fel}, \label{E:Q4}\\
\tau\hy\moQa(A||B)=\tau \ler{\exp({\alpha\log A+(1-\alpha)\log B})} \label{E:Q5}
\end{gather}
for any $A,B\in \App$.

As mentioned above, we will essentially need the concept of Jordan *-isomorphisms between $C^*$-algebras $\A, \Am$. The map $J:\A \to \Am$ is called a Jordan *-isomorphism if it is a bijective linear transformation which has the properties that $J(XY+YX)=J(Y)J(X)+J(X)J(Y)$ and $J(X^*)=J(X)^*$ hold for any $X,Y\in \A$. Those maps are of fundamental importance for several reasons.
For example, they are the basic isomorphisms (symmetries) in the algebraic approach to quantum theory initiated by Segal, see \cite{Seg}. 

In what follows we will see that the studied transformations turn to be of similar forms. In order to simplify the formulations of our results we introduce the following concept.

\begin{definition}
Let $\A,\B$ be $C^*$-algebras with faithful traces $\tau,\omega$, respectively.
We say that a map $\phi$ between the positive definite cones $\App,\Bpp$ or between the density spaces $\DtAi, \DtAmi$ is of {\it the standard form} if there are a central element $C\in \Bpp$ and a Jordan *-isomorphism $J:\A\to \B$ such that
the identity
$\phi(.)=CJ(.)$
holds on the domain of $\phi$ and, moreover, we have $\omega(CJ(X))=\tau(X)$, $X\in \A$.
\end{definition}

In our first main result which follows we describe the structure of all surjective maps between the positive definite cones of $C^*$-algebras with faithful traces which preserve any of the quantum R\'enyi relative entropy related quantities \eqref{E:Q1}-\eqref{E:Q5}. The maps under consideration are kinds of symmetries between those cones.
Our result says that all those maps are of the standard form, they all originate from Jordan *-isomorphisms between the underlying algebras multiplied by central positive invertible elements. It might be worth mentioning the somewhat surprising fact that we do not assume but get it for free that those preservers are automatically linear and even multiplicative to some extent.

\begin{theorem}\label{T:Cone}
Let $\A, \B$ be $C^*$-algebras with faithful traces $\tau,\omega$, respectively, and let $\alpha, z$ be positive numbers, $\alpha\neq 1$. Let $\phi:\App\to \Ampp$ be a surjective map. Then $\phi$ satisfies
\begin{equation}\label{E:1}
\omega\hy\Qaz(\phi(A)||\phi(B))=\tau\hy\Qaz(A||B), \quad A,B\in \App
\end{equation} 
if and only if it of the standard form.

Analogous assertions are valid for all other quantities in \eqref{E:Q1}-\eqref{E:Q5}, for every positive number $\alpha$ different from 1 with the exception of the quantity in \eqref{E:Q4}, where we need to assume that $\alpha \leq 2$.
\end{theorem}

As a corollary, we easily obtain the following description of the structure of maps between density spaces of $C^*$-algebras preserving the different types of quantum R\'enyi relative entropy.

\begin{corollary}\label{C:main}
Let $\A, \B$ be $C^*$-algebras with faithful traces $\tau,\omega$, respectively, and let $\alpha, z$ be positive real numbers, $\alpha\neq 1$. Assume that $\phi:\DtAi\to \DtAmi$ is a surjective map. Then $\phi$ satisfies
\begin{equation*}\label{E:3}
\omega\hy\Daz(\phi(A)||\phi(B))=\tau\hy\Daz(A||B), \quad A,B\in \DtAi
\end{equation*} 
if and only if $\phi$ is of the standard form. 

Analogous assertions are valid for all other quantities in \eqref{E:D1}-\eqref{E:D5}, for every positive number $\alpha$ different from 1 with the exception of \eqref{E:D4}, where we need to assume that $\alpha \leq 2$.
\end{corollary}

In the second main result which follows we show that the above defined quantum R\'enyi relative entropies are essentially different in the following sense: a density space equipped with one such relative entropy can be transformed by any map onto another density space equipped with a different such quantum relative entropy only in the trivial case of commutative algebras, i.e., only in the case of classical, non-quantum systems. The precise statement reads as follows.

\begin{theorem}\label{T:commut}
Let $\A, \B$ be $C^*$-algebras with faithful traces $\tau,\omega$, respectively, and let $\phi:\DtAi\to \DtAmi$ be a surjective map. Assume that $\phi$ satisfies
\begin{equation*}\label{E:13}
\omega\hy\moDa((\phi(A)||\phi(B))=\tau\hy\Daz(A||B), \quad A,B\in \DtAi.
\end{equation*} 
Then the algebras $\A, \B$ are necessarily commutative in which case all  considered types of quantum R\'enyi relative entropy coincide and hence Corollary \ref{C:main} applies and provides the form of $\phi$, in which the corresponding Jordan *-isomorphism is of course necessarily an
 algebra *-isomorphism.

Analogous assertions hold for all other pairs of different quantum R\'enyi relative entropies listed in \eqref{E:D1}-\eqref{E:D5} and for every positive number $\alpha$ different from 1 with the only restriction that concerning the quantity in \eqref{E:D4} we need to assume that $\alpha\leq 2$. 
\end{theorem}

We complete the above results with some additional related ones concerning other fundamental concepts of quantum relative entropy. They are the Umegaki relative entropy and the Belavkin-Staszewski relative entropy which, in our present context, are defined as follows. For any $C^*$-algebra $\A$ with faithful trace $\tau$, 
the Umegaki relative entropy on $\DtAi$ is defined by
\begin{equation}\label{E:U}
S_U^{\tau}(A||B)=\tau\ler{A(\log A-\log B)} 
\end{equation}
for all $A,B\in \DtAi$,
while the Belavkin-Staszewski relative entropy is defined by
\begin{equation}\label{E:BS}
S_{BS}^{\tau}(A||B)=\tau\ler{A\log (A^{\fel}B^{-1}A^{\fel})}
\end{equation}
for all $A,B\in \DtAi$.
Their connection to the quantum R\'enyi relative entropies is the following. In the finite dimensional setting, where $\A$ is the algebra of all $n\times n$ complex matrices and $\tau$ is the usual trace, it is well-known that the conventional R\'enyi relative entropy as well as the minimal R\'enyi relative entropy tends to the Umegaki relative entropy as $\alpha \to 1$. In fact, the same is true for the general $\alpha-z$-R\'enyi relative entropy which was proved in the paper \cite{LM}. The limit of the Mosonyi-Ogawa version of quantum R\'enyi relative entropy is again the Umegaki relative entropy as $\alpha \to 1$, see \cite{MosOga}. Finally, the limiting case of the maximal R\'enyi relative entropy is the Belavkin-Staszewski relative entropy, cf. 4.2.3 in \cite{Toma}.

In Theorem 1 in \cite{ML17a} we described the structure of all bijective maps between the positive definite cones of $C^*$-algebras which preserve the Umegaki relative entropy. (In fact, there we considered an even more general numerical quantity, the so-called quasi-entropy that involves a parameter, namely an invertible element of the underlying algebra which is the identity in our present case.) The proof of that result is very much different from the proof of our Theorem \ref{T:Cone} here. One can easily see that the method of the proof of Corollary \ref{C:main} above can be used to derive the following result from Theorem 1 in \cite{ML17a} on maps respecting the Umegaki relative entropy between density spaces of $C^*$-algebras.

\begin{theorem}\label{T:Udens}
Let $\A, \B$ be $C^*$-algebras with faithful traces $\tau,\omega$, respectively, and let $\phi:\DtAi\to \DtAmi$ be a surjective map. Then $\phi$ preserves the Umegaki relative entropy, i.e., it satisfies
\begin{equation*}\label{E:12}
S_U^{\omega}(\phi(A)||\phi(B))=S_U^{\tau}(A||B), \quad A,B\in \DtAi
\end{equation*} 
if and only if $\phi$ is of the standard form.
\end{theorem}

The structure of maps preserving the Belavkin-Staszewski relative entropy is again the same as we can see in the following theorem. 

\begin{theorem}\label{T:BSdens}
Let $\A, \B$ be $C^*$-algebras with faithful traces $\tau,\omega$, respectively, and $\phi:\DtAi\to \DtAmi$ be a surjective map. Then $\phi$ preserves the Belavkin-Staszewski relative entropy, i.e., it satisfies
\begin{equation}\label{E:33}
S_{BS}^{\omega}(\phi(A)||\phi(B))=S_{BS}^{\tau}(A||B), \quad A,B\in \DtAi
\end{equation} 
if and only if $\phi$ is of the standard form.
\end{theorem}

After describing the symmetry groups of density spaces with respect to the Umegaki and Belavkin-Staszewski relative entropies, which turn to be the same as the symmetry groups with respect to any sorts of quantum R\'enyi relative entropies above,
we conclude the paper with the following analogue of Theorem \ref{T:commut}.

\begin{theorem}\label{T:UBS}
Let $\A, \B$ be $C^*$-algebras with faithful traces $\tau,\omega$, respectively, and let $\phi:\DtAi\to \DtAmi$ be a surjective map which satisfies
\begin{equation*}\label{E:42}
S_{U}^{\omega}(\phi(A)||\phi(B))=S_{BS}^{\tau}(A||B), \quad A,B\in \DtAi.
\end{equation*}
Then the algebras $\A, \B$ are necessarily commutative in which case the Umegaki and the Belavkin-Staszewski relative entropies coincide and hence Theorem \ref{T:Udens} or \ref{T:BSdens} applies and provides the form of $\phi$, in which the corresponding Jordan *-isomorphism is necessarily an algebra *-isomorphism.
\end{theorem}

\section{Proofs and some further results}

In this section we present the proofs of our results formulated above. Moreover, we also present some additional statements, Theorem  \ref{T:homo}, Theorem \ref{T:pot} and Theorem \ref{T:Conedivided} what we obtain on the way. Let us tell in advance that our basic idea in the proofs is simple and can be formulated as follows. We show that our maps under considerations are closely related to certain order isomorphisms. We describe the forms of those isomorphisms and then obtain the desired results. However, the realization of this simple idea, as we will see, is quite complicated. 

In the first part of the preparations we present characterizations of the order (the usual one among self-adjoint elements in $C^*$-algebras) in terms of the various quantum  R\'enyi relative entropies. In the arguments of several such characterizations we will use the following auxiliary result. If $\A$ is a $C^*$-algebra, then $\As$ stands for the space of all of its self-adjoint elements.

\begin{lemma}\label{L:6}
Let $\A$ be a $C^*$-algebra with a faithful trace $\tau$, and let $f:[a,b] \to \R$ be a continuously differentiable function. Pick $X,Y\in \As$ such that $\sigma(X+sY)\subset [a,b]$ holds for all $s$ from a nondegenerate real interval $[c,d]$. Then we have
\begin{equation}\label{E:6}
\frac{d}{ds} \tau \ler{f(X+sY)}_{|{s=t}}=\tau (f'(X+tY)Y)=\tau(Y^\fel f'(X+tY) Y^\fel), \quad t\in [c,d].
\end{equation}
In particular, if $f$ is increasing, then for any $A,B\in \As$ with $\sigma(A),\sigma(B)\subset [a,b]$ and $A\leq B$, we have $\tau(f(A))\leq \tau(f(B))$. Moreover, if $f'$ is everywhere positive on $[a,b]$ and for a given pair $A,B\in \As$ we have $\sigma(A),\sigma(B)\subset [a,b]$, $A\leq B$ and $\tau(f(A))= \tau(f(B))$, then it follows that $A=B$. 
\end{lemma}

\begin{proof}
In the case of matrix algebras, the formula \eqref{E:6} was given in Theorem 11.9 in \cite{Petz}. In our general setting, choose $X,Y$ as above, i.e., let $X,Y\in \As$ be such that $\sigma(X+sY)\subset [a,b]$ holds for all $s$ from a nondegenerate interval $[c,d]$. One can easily verify that 
\eqref{E:6}
holds for any power function $f$ with nonnegative integer exponent. It then follows that it also holds whenever $f$ is a polynomial. We finally apply polynomial approximation. We choose a sequence $(p_n)$ of polynomials such that $p_n \to f$ and $p_n'\to f'$ uniformly on $[a,b]$. We refer to a well-known result from calculus stating that if a sequence of continuously differentiable functions and also the sequence of its derivatives converge uniformly, then the limit of the former sequence is continuously differentiable and its derivative equals the limit of the latter sequence. Applying that result and using the boundedness of trace functional $\tau$ (which follows from its positivity), the validity of the equality \eqref{E:6} follows easily.

Assume now that $f$ is increasing, $A,B\in \As$ are such that $\sigma(A),\sigma(B)\subset [a,b]$ and $A\leq B$. By \eqref{E:6} we have 
\begin{equation*}
\tau (f(B))-\tau(f(A))=\int_0^1 \tau \ler{ (B-A)^\fel f'(A+t(B-A))(B-A)^\fel}  \,dt.
\end{equation*}
Here the function what we integrate is a continuous nonnegative function. It follows that $\tau (f(B))-\tau(f(A))\geq 0$. 
Suppose next that $f'$ is everywhere positive on $[a,b]$ and $\tau (f(B))-\tau(f(A))=0$. We deduce from the above displayed equality that $\tau ( (B-A)^\fel f'(A+t(B-A))(B-A)^\fel)=0$, $t\in [0,1]$. This implies $(B-A)^\fel f'(A+t(B-A))(B-A)^\fel=0$, $t\in [0,1]$ and since the middle term here is positive invertible, we easily conclude that $B-A=0$.
\end{proof}

Now, our first characterization of the order in terms of quantum R\'enyi relative entropies is the following.

\begin{lemma}\label{L:1}
Let $\alpha\in ]0,\infty[$ be a real number, $\A$ be a $C^*$-algebra with a faithful trace $\tau$, and select $A,B\in \App$. We have $A\leq B$ if and only if $\tau\ler{(XAX)^\alpha} \leq \tau\ler{(XBX)^\alpha}$ holds for all $X\in \App$.

Therefore, for any given real number $z>0$, we have that 
$A^{\frac{\alpha}{z}}\leq B^{\frac{\alpha}{z}}$ if and only if $\tau\hy\Qaz(A||X)\leq \tau\hy\Qaz(B||X)$ is valid for all $X\in \App$, and we have $A^{\frac{1-\alpha}{z}}\leq B^{\frac{1-\alpha}{z}}$ if and only if $\tau\hy\Qaz(X||A)\leq \tau\hy\Qaz(X||B)$ holds for all $X\in \App$.
\end{lemma}

\begin{proof}
By Gelfand-Naimark theorem we may assume that $\A$ is a $C^*$-subalgebra of the full operator algebra $B(H)$ over some complex Hilbert space $H$ containing the identity $I$.

The necessity part of the first statement follows from Lemma \ref{L:6}. 

Assume now, that $\tau\ler{(XAX)^\alpha} \leq \tau\ler{(XBX)^\alpha}$ holds for all $X\in \App$. By the continuity of the power function and the trace, we have the same inequality also for any $X\in \Ap$.  Consider the self-adjoint element $B-A=S$  of $\A$ whose spectrum is contained in the interval $[a,b]\subset \R$. Let $f$ be any continuous nonnegative function on $[a,b]$ with values in $[0,1]$ which is zero  in the positive part of $[a,b]$. Then $f(S)\in  \Ap$, and we have $f(S)Sf(S)\leq 0$ implying $f(S)Bf(S)\leq f(S)Af(S)$. Assume $\alpha \geq 1$. Then we apply Lemma \ref{L:6} and obtain $\tau\ler{(f(S)Bf(S))^\alpha} \leq \tau\ler{(f(S)Af(S))^\alpha}$. However, by the assumption, we have the opposite inequality,  too. By Lemma \ref{L:6} we deduce $f(S)Bf(S)= f(S)Af(S)$. If $\alpha<1$, then we use the operator monotonicity of the power function with exponent $\alpha$ (observe that the function $t\mapsto t^\alpha$ is not differentiable at 0) to obtain $(f(S)Bf(S))^\alpha \leq (f(S)Af(S))^\alpha$ and then from the equality $\tau\ler{(f(S)Bf(S))^\alpha} = \tau\ler{(f(S)Af(S))^\alpha}$ we easily deduce that $f(S)Bf(S)= f(S)Af(S)$.
Therefore, we have $f(S)Sf(S)=0$. 
Now, we can choose a sequence $(f_n)$ of such functions $f$ which pointwise converges to the indicator function of the subinterval $[a,0[$ of $[a,b]$. We obtain that the corresponding operator sequence $(f_n(S))$ strongly converges to the spectral projection $P$ of $S$ corresponding to the interval $[a,0[$. It follows that $PSP=0$ which implies that $S$ is positive, i.e., we have $A\leq B$.

After this, the first statement in the last sentence of the lemma is apparent. As for the second one, observe that for any $C,D\in \App$ we have that $CD^2 C$ and $DC^2 D$ are unitarily equivalent (consider the polar decomposition of $DC$). It follows that for any $A,X\in \App$ we have 
\begin{equation*}
\tau\hy\Qaz(X||A)=\tau \ler{A^{\frac{1-\alpha}{2z}}X^{\frac{\alpha}{z}}A^{\frac{1-\alpha}{2z}}}^{z}=
\tau \ler{X^{\frac{\alpha}{2z}}A^{\frac{1-\alpha}{z}}X^{\frac{\alpha}{2z}}}^{z}
\end{equation*}
and then we can apply the first statement of the lemma.
\end{proof}

Next, we characterize the order on the positive definite cone by the quantity $\tau\hy\maxQa(.||.)$. In order to do that, we need the following observation.

\begin{lemma}\label{L:20}
Let $f:]0,\infty[\to \mathbb R$ be a continuous function. Assume that
the function $g:[0,\infty[\to \mathbb R$ defined by $g(t)=tf(t)$ for $t>0$ and $g(0)=0$ is continuous on $[0,\infty[$.
Let $H$ be a complex Hilbert space and $A\in B(H)$ be a positive operator with closed range $K
\subset H$. For any invertible positive operator $B\in B(H)$, we deduce that $AB^2 A$ is a positive invertible operator on $K$ and we have
\begin{equation}\label{E:37}
Af(AB^2A)A=
B^{-1} g(BA^2 B)B^{-1}.
\end{equation}
\end{lemma}

\begin{proof}
Clearly, the restriction of $A$ to $K$ is a positive invertible operator. It follows easily that the restriction of $AB^2A$ to $K$ is also invertible. The formula \eqref{E:37} holds whenever $f$ is a polynomial.
Uniformly approximating $f$ on the spectrum of the operator $AB^2A$ (when it is considered on the subspace $K$), and taking into consideration that the spectrum of $AB^2A$ and $BA^2B$ may differ only in one single element, namely 0, we obtain the general conclusion.
\end{proof}

Using the previous observation we obtain the following characterization of the order.

\begin{lemma} \label{L:BS}
Let $f:]0,\infty[\to \mathbb R$ be a strictly increasing continuous function which is also operator monotone. 
Assume that the function $g:[0,\infty[\to \mathbb R$ defined by $g(t)=tf(t)$ for $t>0$ and $g(0)=0$ is continuous on $[0,\infty[$. Let $\A$ be a $C^*$-algebra with faithful trace $\tau$ and select $A,B\in \App$. The following assertions are equivalent:
\begin{itemize}
\item[(i)] $A^2\leq B^2$;
\item[(ii)] $\tau(Xf(XA^2 X)X)\leq  \tau(Xf(XB^2 X)X)$ holds for all $X\in \App$.
\end{itemize}
\end{lemma}

\begin{proof}
The implication (i)$\Rightarrow$(ii) is trivial, it follows from the operator monotonicity of $f$ and the monotonicity of the trace $\tau$.

Assume now that (ii) holds. As we have already done in the proof of Lemma \ref{L:1}, assume that $\A$ is a $C^*$-subalgebra of $B(H)$ for some complex Hilbert space $H$ containing the identity $I$. 

Consider the spectral measure of the self-adjoint operator $B^2-A^2\in \A$ on the Borel sets of the real line. Let $P_0$ be the spectral measure of the interval $]-\infty, 0[$ and $P_n$ be the spectral measure of $]-\infty,-1/n]$, $n\in \mathbb N$. For any positive integer $m$, we consider the continuous function $h_m$ which is 1 on the interval $]-\infty , -1/m]$, has derivative $-m$ on $]-1/m,0[$, and equals zero on $[0,\infty[$. Define $X_m\in \Ap$ by $X_m=h_m(B^2-A^2)$. We have that $Y_{nm}=X_mP_n=P_nX_m$ is a positive operator with closed range. Clearly, $Y_{nm} B^2Y_{nm}\leq Y_{nm} A^2Y_{nm}$, hence, by the previous lemma, we have that
\begin{equation*}
B^{-1}g(B{Y_{nm}}^2B)B^{-1}=
Y_{nm}f(Y_{nm} B^2 Y_{nm})Y_{nm}\leq
Y_{nm}f(Y_{nm} A^2 Y_{nm})Y_{nm}=
A^{-1}g(A{Y_{nm}}^2A)A^{-1}
\end{equation*}
holds for all $n,m\in \mathbb N$. The double sequence $(Y_{nm})_{nm}$ is bounded and we see that for fixed $m$,  when $n$ tends to infinity, $(Y_{nm})$ converges in norm to $X_m$. It follows that $(A{Y_{nm}}^2A)_{nm}$ is also bounded and, for fixed $m$, it converges in norm to $A{X_{m}}^2A$ as $n\to \infty$. Consequently, $g(A{Y_{nm}}^2A)$ converges in norm to $g(A{X_{m}}^2A)$ as $n\to \infty$. Obviously, the same holds for $B$ in the place of $A$, too. Therefore, we obtain
\begin{equation}\label{E:31}
B^{-1}g(B{X_m}^2B)B^{-1}\leq
A^{-1}g(A{X_m}^2A)A^{-1}, \quad m\in \mathbb N.
\end{equation}
On the other hand, by our condition (ii) and the identity \eqref{E:37}, we deduce that
\begin{equation*} 
\tau(B^{-1}g(B{X}^2B)B^{-1})=\tau(Xf(XB^2X)X)\geq \tau(Xf(XA^2X)X)=
\tau(A^{-1}g(A{X}^2A)A^{-1})
\end{equation*}
holds for all $X\in \App$. Using continuity arguments, we  can infer that the inequality holds also for any $X\in \Ap$ and hence we obtain
\begin{equation*}
\tau(B^{-1}g(B{X_m}^2B)B^{-1})\geq
\tau(A^{-1}g(A{X_m}^2A)A^{-1}), \quad m\in \mathbb N.
\end{equation*}
By \eqref{E:31}, it follows that
\begin{equation*}
\tau(B^{-1}g(B{X_m}^2B)B^{-1})=
\tau(A^{-1}g(A{X_m}^2A)A^{-1}), \quad m\in \mathbb N
\end{equation*}
and, by \eqref{E:31} again and using the faithfulness of $\tau$, we infer that
\begin{equation*}
B^{-1}g(B{X_m}^2B)B^{-1}=
A^{-1}g(A{X_m}^2A)A^{-1}, \quad m\in \mathbb N.
\end{equation*}
We know that the bounded continuous real functions are strongly continuous (see, e.g., 4.3.2. Theorem in \cite{Mur}). Clearly, the sequence $(X_m)$ is bounded and strongly converges to $P_0$.
It follows that, taking strong limits in the equality above, we can infer that 
\begin{equation*}
B^{-1}g(B{P_0}^2B)B^{-1}=
A^{-1}g(A{P_0}^2A)A^{-1}.
\end{equation*}
Using \eqref{E:37} again, we get that
\begin{equation*}
P_0f(P_0B^2P_0)P_0=P_0f(P_0A^2P_0)P_0.
\end{equation*}
This means that, on the range of $P_0$, we have the identity $f(P_0B^2P_0)=f(P_0A^2P_0)$ which, by the strict monotonicity (and hence invertibility) of $f$, implies $P_0B^2P_0=P_0A^2P_0$. We infer
$P_0(B^2-A^2)P_0=0$ and, by the particular choice of $P_0$, it follows that $B^2\geq A^2$. 
\end{proof}

Using the above result we can easily get the following.

\begin{lemma}\label{L:7}
Let $\alpha\in ]0,2]$, $\alpha\neq 1$ be a real number, $\A$ be a $C^*$-algebra with a faithful trace $\tau$, and select $A,B\in \App$. For $0<\alpha<1$ we have $A\leq B$ if and only if 
$\tau\hy\maxQa(A||X)\leq \tau\hy\maxQa(B||X)$ holds for all $X\in \App$, while for $1<\alpha \leq 2$ we have $A\leq B$ if and only if $\tau\hy\maxQa(X||B)\leq \tau\hy\maxQa(X||A)$ holds for all $X\in \App$.
\end{lemma}

\begin{proof}
First assume that $\alpha<1$. Applying Lemma \ref{L:BS} for the operator monotone function $f(t)=-t^{-\alpha}$, $t>0$, we obtain that $A^2\leq B^2$ is valid if and only if  $\tau(X(X^{-1}B^{-2} X^{-1})^\alpha X)\leq  \tau(X(X^{-1}A^{-2} X^{-1})^\alpha X)$ holds for all $X\in \App$. One can easily conclude from this that
we have $A\leq B$ if and only if $\tau\hy\maxQa(A||X)\leq \tau\hy\maxQa(B||X)$ holds for all $X\in \App$.

Assume now that $1<\alpha \leq 2$. By Lemma \ref{L:BS} again, we have that $A^2\leq B^2$ is valid if and only if $\tau(X(XA^2 X)^{\alpha -1}X)\leq  \tau(X (XB^2 X)^{\alpha-1} X)$ holds for all $X\in \App$. On the other hand, by Lemma \ref{L:20} we have
$$ 
X(XA^{-2} X)^{\alpha -1}X=A(A^{-1} X^2A^{-1})^\alpha A
$$
for any pairs $A,X\in \App$. From these we can easily conclude that 
$A\leq B$ if and only if $\tau\hy\maxQa(X||B)\leq \tau\hy\maxQa(X||A)$ holds for all $X\in \App$.
\end{proof}

The next result gives a characterization of the order in terms of the quantity $\tau\hy\moQa(.||.)$. 

\begin{lemma}\label{L:4}
Let $\A$ be a $C^*$-algebra with a faithful trace $\tau$. Pick $T,S\in \As$. We have $T\leq S$ if and only if
$\tau(\exp(T+X))\leq \tau(\exp(S+X))$ holds for all $X\in \As$. 

In particular, for any $\alpha\in ]0,\infty[$, $\alpha\neq 1$ and for arbitrary $A,B\in \App$, we have $\log A\leq \log B$ if and only if $\tau\hy\moQa(A||X) \leq \tau\hy\moQa(B||X)$ holds for all $X\in \App$. Moreover, if $\alpha>1$, then for any $A,B\in \App$ we have $\log A\leq \log B$ if and only if $\tau\hy\moQa(X||B) \leq \tau\hy\moQa(X||A)$ holds for all $X\in \App$. 
\end{lemma}

\begin{proof}
The necessity part of the first statement follows from Lemma \ref{L:6}.

As for the sufficiency, assume that $\tau(\exp(T+X))\leq \tau(\exp(S+X))$ holds for all $X\in \As$. Denote $C=S-T$ and write $X-T$ in the place of $X$. We have $\tau(\exp(X))\leq \tau(\exp(C+X))$, $X\in \As$. Let $D=\exp (C)$. Then for every $E\in \App$ which commutes with $D$ we have that $\log E$ commutes with $C$ and, choosing $X=\log E$, we obtain $\tau(E)\leq \tau (DE)$. An argument similar to what was applied in the sufficiency part of the proof of Lemma \ref{L:1} gives us that $I\leq D$. This implies $C\geq 0$, i.e., $T\leq S$.

The remaining part of the lemma is just obvious.
\end{proof}

We next collect some useful properties of Jordan *-isomorphisms that we will use in what follows. First, any Jordan *-isomorphism $J:\A\to \Am$ satisfies
\begin{equation*}\label{E:J1}
J(XYX)=J(X)J(Y)J(X), \quad X,Y\in \A,
\end{equation*}
and hence
\begin{equation}\label{E:J2}
J(X^n)=J(X)^n, \quad X\in \A
\end{equation}
holds for every nonnegative integer $n$, see 6.3.2 Lemma in \cite{Pal}. In particular, $J$ is unital meaning that $J$ sends the identity to the identity. Since $J$ is clearly positive (in fact, it preserves the order between self-adjoint elements in both directions), it is bounded. Indeed, more is true: $J$ is an isometry with respect to the $C^*$-norm. 
By Proposition 1.3 in \cite{Sou}, $J$ preserves invertibility,  namely we have
\begin{equation*}\label{E:J3}
J(X^{-1})=J(X)^{-1}
\end{equation*}
for every invertible element $X\in \A$.
It follows that $J$ preserves the spectrum and, using continuous function calculus, from \eqref{E:J2} we deduce that 
\begin{equation*}\label{E:J4}
J(f(X))=f(J(X))
\end{equation*}
holds for any self-adjoint element $X\in \As$ and continuous real function $f$ on the spectrum of $X$. 

We continue the preparations and recall the definition of the Thompson metric (or Thompson part metric). In fact, it can be defined in a rather general setting involving normed linear spaces and certain closed cones, see \cite{Tho63}. In the case of a $C^*$-algebra $\A$, that general definition of the Thompson metric $d_T$ on the positive definite cone $\App$ reads as follows
\begin{equation}\label{E:Tho}
d_T (A, B) = \log \max \{M (A/B),M (B/A)\}, \quad A,B\in \App,
\end{equation}
where $M(X/Y)= \inf\{ t>0 \, : \, X\leq t Y\}$ for any $X,Y\in \App$. It is easy to see that $d_T$ can also be rewritten as
\[
d_T (A,B)=\left\| \log \left(A^{-\fel} B A^{-\fel}\right) \right\|, \quad A,B\in \App.
\]
Here and in what follows $\|.\|$ denotes the $C^*$-norm on $\A$.

The structure of surjective Thompson isometries is known and it was described in our paper \cite{ML14a}. By Theorem 9 in \cite{ML14a}, we have that for given $C^*$-algebras $\A,\Am$ and surjective Thompson isometry $\phi:\App\to \Ampp$, 
there are a central projection $P$ in $\Am$ and a Jordan *-isomorphism $J:\A\to \Am$  such that $\phi$ is of the form
\begin{equation}\label{E:36}
\phi(A) = \phi(I)^{1/2}\left(PJ(A)+(I-P)J(A^{-1})\right)\phi(I)^{1/2}
, \quad A\in \App.
\end{equation} 
(We remark that the converse statement is also true, any map between the positive definite cones $\App,\Ampp$ of the form \eqref{E:36} is necessarily a surjective Thompson isometry.) The crucial observation we make below concerns positive homogeneous order isomorphisms. If $\phi:\App \to \Bpp$ is a bijective map such that for any $A,B\in \App$ we have $A\leq B$ if and only if $\phi(A)\leq \phi(B)$, then $\phi$ is called an order isomorphism. Moreover, we say that $\phi$ is positive homogeneous if $\phi(tA)=t\phi(A)$ holds for all $A\in \App$ and real number $t>0$.

Theorem 9 in \cite{ML14a} immediately gives us the following.

\begin{theorem}\label{T:cruc}
Let $\A,\B$ be $C^*$-algebras. The map $\phi:\A \to \B$ is a positive homogeneous order isomorphism if and only if it is of the form
\[
\phi(A)=CJ(A)C, \quad A\in \App,
\]
where $C\in \Bpp$ and $J:\A\to \B$ is a Jordan *-isomorphism.
\end{theorem}

Beside surjective Thompson isometries we will also need to consider surjective  dilations (homotheties) with respect to the Thompson metric.
In the proof of the corresponding result Theorem \ref{T:homo} and also in the proof of Theorem \ref{T:pot}, we will use a general Mazur-Ulam type result of ours, see Theorem 3 in \cite{MolGMU}. For the sake of completeness, below we formulate that general result but in a somewhat weaker form which is however just appropriate for our present aims.

First we need a concept. Let $X$ be a set equipped with a binary operation $\diamond$ which satisfies the following conditions:
\begin{itemize}
\item[(a1)]
$a\diamond a=a$ holds for every $a\in X$;
\item[(a2)]
$a\diamond (a\diamond b)=b$ holds for any $a,b\in X$;
\item[(a3)]
the equation $x\diamond a=b$ has a unique solution $x\in X$ for any given $a,b\in X$.
\end{itemize}
Then we call the pair $(X,\diamond)$ a point-reflection geometry.
A trivial example for such a structure is any linear space equipped with the operation $a\diamond b =2a-b$. A quite nontrivial example is when we consider the positive definite cone $\App$ of a $C^*$-algebra $\A$ equipped with the operation $a\diamond b =ab^{-1}a$, $a,b\in \App$ (see \cite{MolGMU}, the discussion after Definition 1). Now, the general Mazur-Ulam theorem we need reads as follows.
 
\begin{theorem}[cf. Theorem 3 in \cite{MolGMU}]\label{T:GMU}
Let $(X, \diamond)$, $(Y, \star)$ be point-reflection geometries equipped with metrics $d,\rho$, respectively, such that
\begin{itemize}
\item[(b1)]
$d(a\diamond x, a\diamond y)=d(x,y)$ holds for all $a,x,y\in X$ and, similarly,
\item[(b1')]
$\rho(a'\star x', a'\star y')=\rho(x',y')$ is valid for all $a',x',y'\in Y$;
\item[(b2)]
there exists a constant $K>1$ such that $d(x,y\diamond x)\geq K d(x, y)$ holds for every $x,y\in X$.
\end{itemize}
If $\phi:X\to Y$ is a surjective isometry, i.e.,
a surjective map which satisfies
\[
\rho(\phi(x),\phi(y))=d(x,y), \quad x,y\in X, 
\]
then we have that $\phi$ is an isomorphism in the sense that
\[
\phi(x\diamond y)=\phi(x) \star \phi(y), \quad x,y\in X.
\]
\end{theorem}  

The following interesting result says that the existence of a non-isometric surjective dilation (homothety) between the positive definite cones of $C^*$-algebras with respect to the Thompson metric implies that the underlying algebras are necessarily commutative. More precisely, we have the following statement.

\begin{theorem}\label{T:homo}
If $\A$ and $\Am$ are $C^*$-algebras and there is a surjective map $\phi:\App \to \Ampp$ such that 
\begin{equation}\label{E:70}
d_T(\phi(A),\phi(B))=\gamma d_T(A,B), \quad A,B\in \App
\end{equation}
holds with some positive real number $\gamma$ different from 1, then the algebras $\A,\Am$ are necessarily commutative.
\end{theorem}

Note that the first part of the proof goes along the lines of the proof of Theorem 9 in \cite{ML14a}.

\begin{proof}
Assume that we have a surjective map $\phi$ satisfying \eqref{E:70}. Clearly, all maps of the form $B\mapsto TBT^*$, with any invertible $T\in \B$ are Thompson isometries of $\Bpp$. Therefore, considering the transformation $\phi(I)^{\mfel}\phi(.)\phi(I)^{\mfel}$, we can and do assume that $\phi$ is unital, it sends the unit to the unit. 

We can apply our general Mazur-Ulam theorem, Theorem \ref{T:GMU}, for the pair $d_T,\gamma d_T$ of metrics on the point-reflection geometries $\App, \Bpp$, respectively. (To see that all the conditions in Theorem \ref{T:GMU} are satisfied, we refer the reader to the proof of Proposition 13 in \cite{MolGMU}.)  We infer that 
\begin{equation*}\label{E:32}
\phi(AB^{-1}A)=\phi(A)\phi(B)^{-1}\phi(A), \quad A,B\in \App.
\end{equation*}
Since $\phi$ sends the identity to the identity, it follows that $\phi(B^{-1})=\phi(B)^{-1}$, $B\in \App$ and thus $\phi$ fulfills
\begin{equation}\label{E:P1}
\phi(ABA)=\phi(A)\phi(B)\phi(A), \quad A,B\in \App.
\end{equation}
The topology of the Thompson metric coincides with the topology of the $C^*$-norm on $\App$ (see Proposition 13 in \cite{MolGMU} for a more general statement). Since $\phi$ is trivially continuous with respect to the Thompson metric, it is continuous with respect to the $C^*$-norm. It follows that we have
$\phi(A^t)=(\phi(A))^t$ for any $A\in \App$ and real number $t$. In fact, using \eqref{E:P1}, one can first prove this identity for integers, next for rationals and finally, using continuity, for all reals. Define $F(S)=\log \phi(\exp(S))$, $S\in \As$. We know from \cite{ML14a} (see p. 166 there) that the formula 
\begin{equation*}
\frac{d_T(\exp({tS}),\exp({tR}))}{t} \to \|S-R\| \quad \text{ as } t\to 0
\end{equation*}
holds for all $S,R\in \As$. Since
\[
\begin{gathered}
d_T(\exp(tF(S)),\exp(tF(R)))=
d_T((\phi(\exp(S)))^t,(\phi(\exp(R)))^t)=\\
d_T(\phi(\exp(tS)),\phi(\exp(tR)))=
\gamma d_T(\exp(tS),\exp(tR)), \quad t>0,
\end{gathered}
\]
hence we deduce
\begin{equation*}
\|F(S)-F(R)\|=\gamma \|S-R\|, \quad S,R\in \As.
\end{equation*}
We know that $\phi(I)=I$ and it implies that $F(0)=0$. We have that $(1/\gamma)F$ is a surjective isometry between the normed real linear spaces $\As$ and $\Ams$. The classical Mazur-Ulam theorem asserts that any surjective isometry between normed real linear spaces is affine and hence it is a surjective linear isometry followed by a translation. Therefore, we obtain that $(1/\gamma)F$ is a surjective linear isometry. The structure of such maps between the self-adjoint parts of $C^*$-algebras was described by Kadison, see Theorem 2 in \cite{Kad}. That result says that $F$ is necessarily of the form
\begin{equation*}
F(S)=\gamma CJ(S), \quad S\in \As,
\end{equation*}
where $C\in \Ams$ is a central symmetry (central self-adjoint unitary) and $J:\A \to \B$ is a Jordan *-isomorphism. Concerning $\phi$, this means that
\begin{equation*}
\phi(A)=\exp\ler{\gamma CJ(\log A)}, \quad A\in \App.
\end{equation*}

Since Jordan *-isomorphisms (as well as their inverses) send commuting elements to commuting elements  (see, e.g., 6.3.4 Theorem in \cite{Pal}), it follows that there is a central symmetry $D\in \A$ such that $J(D)=C$ and hence we easily have
\[
\phi(A)=J(\exp(\gamma D\log A)), \quad A\in \App.
\]
Since Jordan *-isomorphisms, when restricted to positive definite cones, are clearly Thompson isometries, hence, by \eqref{E:70} we have
\[
d_T(\exp(\gamma D\log A),\exp(\gamma D\log B))=\gamma d_T(A,B), \quad A,B\in \App.
\]
It is not difficult to check that with the central symmetry $D$, the transformation $A\mapsto  \exp(D\log A)$ is also a Thomson isometry. Therefore, from the above displayed formula we infer
\[
d_T(A^{\gamma},B^{\gamma})=\gamma d_T(A,B), \quad A,B\in \App.
\]
This clearly implies the validity of the following identity
\[
\|\log (A^{\gamma}B^{\gamma}A^{\gamma})\|=\|\log (ABA)^{\gamma}\|, \quad A,B\in \App.
\]
Choosing elements $A,B\in \App$ such that $A,B\geq I$, we have
\[
\log \|A^{\gamma}B^{\gamma}A^{\gamma}\|=
\|\log (A^{\gamma}B^{\gamma}A^{\gamma})\|=\|\log (ABA)^{\gamma}\|=\log \|(ABA)^{\gamma}\|, \quad A,B\in \App.
\]
Therefore, for such $A,B\in \App$, it follows that 
\begin{equation}\label{E:71}
\|A^{\gamma}B^{\gamma}A^{\gamma}\| = \|ABA\|^{\gamma}.
\end{equation}
Obviously, multiplying any $A,B\in \App$ by positive scalars, we obtain the above equality for all $A,B\in \App$, too.
Next observe that for any $E,F\in \App$ the following holds: $E\leq F$ if and only if $\|XEX\|\leq \|XFX\|$ is valid for all $X\in \App$. Indeed, only the sufficiency needs proof. Choose $X=F^{\mfel}$. Then we have  $\|F^{\mfel}EF^{\mfel}\|\leq 1$, implying $F^{\mfel}EF^{\mfel}\leq I$ which gives $E\leq F$. Consequently, from \eqref{E:71} we can deduce that the map $B\mapsto B^\gamma$ is an order automorphism of $\App$.

Theorem 2 in \cite{NUW} states (among others) that if a nonconcave continuous increasing numerical function is operator monotone on $\App$, then $\A$ is necessarily commutative. Applying this result we obtain that $\A$ is commutative and since Jordan *-isomorphisms preserve commutativity, it follows that $\B$ is also commutative. The proof is complete. 
\end{proof}

The following is a crucial result in which we extend the statement of Theorem \ref{T:cruc}. 
Here we describe the forms of positive homogeneous surjective maps between positive definite cones which respect  certain pairs of order relations.

\begin{theorem}\label{T:pot}
Let $f,g$ each be either the logarithm function or a power function with a positive exponent defined on the positive real line. Let $\A,\B$ be $C^*$-algebras, $\phi:\App\to \Bpp$ be a surjective positive homogeneous map such that for any $A,B\in \App$ we have $f(A)\leq f(B)$ if and only if $g(\phi(A))\leq g(\phi(A))$.
We can describe the structure of $\phi$ as follows.
\begin{itemize}
\item[(c1)] If $f,g$ are both power functions and $f=g$, then $\phi$ is of the form
\begin{equation*}
\phi(A)= f^{-1}(CJ(f(A))C), \quad A\in \App,
\end{equation*}
where $C\in \Bpp$ and $J:\A\to \B$ is a Jordan *-isomorphism. 
\item[(c2)] If $f,g$ are both power functions and $f\neq g$, then the algebras $\A,\B$ are necessarily commutative and $\phi$ is of the form
\begin{equation*}
\phi(A)= CJ(A), \quad A\in \App,
\end{equation*}
where $C\in \Bpp$ and $J:\A\to \B$ is an algebra *-isomorphism.
\item[(c3)] If $f=g=\log$, then $\phi$ is of the form
\begin{equation*}
\phi(A)=\exp(J(\log A)+X_0), \quad A\in \App,
\end{equation*}
where $J:\A\to \B$ is a Jordan *-isomorphism and $X_0\in \Ams$. 
\item[(c4)] If $f$ is a power function and $g$ is the logarithmic function, then
$\A, \B$ are necessarily commutative and $\phi$ is of the form
\begin{equation*}
\phi(A)=CJ(A), \quad A\in \App
\end{equation*}
where $C\in \Bpp$ and $J:\A\to \B$ is an algebra *-isomorphism.
\end{itemize}
\end{theorem}

\begin{proof}
The injectivity of $\phi$ is obvious. Indeed, from $\phi(A)=\phi(B)$ we obtain $f(A)=f(B)$ which implies $A=B$ for any $A,B\in \App$. 

Assume that both $f,g$ are power functions, $f(t)=t^\gamma$, $g(t)=t^\delta$, $t>0$ holds with some positive real numbers $\gamma,\delta$. Then the bijective map 
$\psi:\App\to \Bpp$ defined by $\psi(A)=\phi(A^{\frac{1}{\gamma}})^\delta$, $A\in \App$ is an order isomorphism meaning that for any $A,B\in \A$ we have $A\leq B$ if and only if $\psi(A)\leq \psi(B)$ holds. Moreover, $\psi(tA)=t^{\frac{\delta}{\gamma}}\psi(A)$ for any $A\in \App$ and positive real number $t$. 
Therefore, we have $A\leq tB$ if and only if $\psi(A)\leq t^{\frac{\delta}{\gamma}}\psi(B)$ and, by the definition of the Thompson metric in \eqref{E:Tho}, it is easy to see that we have
\[
d_T(\psi(A),\psi(B))=\frac{\delta}{\gamma} d_T(A,B), \quad A,B\in \App.
\]
If $\delta=\gamma$, then we obtain that $\psi$ is a surjective Thompson isometry which is positive homogeneous and we can apply Theorem \ref{T:cruc} to deduce that $\psi(A)=CJ(A)C$, $A\in \App$ holds with some $C\in \Bpp$ and Jordan *-isomorphism $J:\A \to \B$. Therefore, we have 
\[
\phi(A)=(CJ(A^\gamma)C)^{\frac{1}{\gamma}}, \quad A\in \App
\]
which proves (c1).
If $\delta\neq \gamma$, then applying Theorem \ref{T:homo}, we obtain that $\A,\B$ are commutative. But then $\phi$ obviously has the property that $A\leq B$ holds if and only if $\phi(A)\leq \phi(B)$. We obtain that $\phi$ is a positive homogeneous Thompson isometry and one can complete the proof of (c2) easily.

Assume next that both $f,g$ are the logarithmic function.
Then the bijection $\phi:\App\to \Bpp$ has the following property: for any $A,B\in \App$ we have $\log A\leq \log B$ if and only if $\log \phi(A)\leq \log \phi(B)$.  Define the bijective map $\psi: T\mapsto \log\phi(\exp(T))$ on $\As$.
It is apparent that $\psi$ is an order isomorphism between the spaces $\As$ and $\Ams$. Moreover, because of the homogeneity of $\phi$, we calculate as follows
\begin{equation*}
\psi(T+tI)=\log \phi(e^t\exp(T))=\log e^t \phi(\exp(T))=
tI+\psi(T), \quad T\in \As, t\in \R.
\end{equation*}
Consequently, for any $T,S\in \As$ and real number $t$, the next equivalences hold true
\begin{equation*}
T-S\leq tI\Leftrightarrow T\leq S+tI \Leftrightarrow
\psi(T)\leq\psi(S+tI)=\psi(S)+tI\Leftrightarrow \psi(T)-\psi(S)\leq tI.
\end{equation*}
It is apparent that for any element $X\in \As$, we have the following formula for its norm:
\begin{equation}\label{E:25}
\| X\|=\max\{ \min\{t\in \R \, :\, X\leq tI\}, \min\{t\in \R \, :\, -X\leq tI\} \}.
\end{equation}
Using this, we obtain that $\psi$ satisfies 
\[
\|\psi(T)-\psi(S)\|=\| T-S\|, \quad T,S\in \As,
\]
i.e., $\psi$
is a surjective isometry between the normed real linear spaces $\As$ and $\Ams$. 
Applying the classical Mazur-Ulam theorem, we obtain that $\psi$ is a surjective linear isometry followed by a translation. 
Using Kadison's result Theorem 2 in \cite{Kad}  again, we have a central symmetry $C\in \B$, a Jordan *-isomorphism $J:\A\to \B$ and an element $X_0$ in $\Ams$ such that
$\log(\phi(\exp(T)))=\psi(T)=CJ(T)+X_0$ holds for all $T\in \As$. It follows that $\phi$ is necessarily of the form
\begin{equation}\label{E:23}
\phi(A)=\exp(CJ(\log A)+X_0), \quad A\in \App.
\end{equation}
Choosing $A=tI$ for any $t>0$ and using \eqref{E:23}, from the identity $\phi(tI)=t\phi(I)$ we deduce that
$\exp((\log t)C+X_0)=\exp((\log t)I+X_0)$. Clearly, this gives us that $C$ is the identity. Therefore, we have
\begin{equation*}
\phi(A)=\exp(J(\log A)+X_0), \quad A\in \App
\end{equation*}
which proves (c3).

Assume finally that $f$ is the power function with exponent $\gamma>0$ and $g$ is the logarithmic function.
For any $A,B\in \App$, the inequality $A^{\gamma}\leq B^{\gamma}$ holds if and only if $\log \phi(A)\leq \log \phi(B)$ is valid. Therefore,
the transformation $\psi: A\mapsto \log \phi(A^{\frac{1}{\gamma}})$ is an order isomorphism between $\App$ and $\Ams$.
We have that
\begin{equation*}
A\leq tB \Leftrightarrow \log \phi(A^{\frac{1}{\gamma}})\leq \log \phi((tB)^{\frac{1}{\gamma}})=
\log {t}^{\frac{1}{\gamma}}\phi(B^{\frac{1}{\gamma}})
\Leftrightarrow
\psi(A)\leq \ler{\frac{1}{\gamma} \log t}I +\psi(B)
\end{equation*}
for all $A,B\in \App$ and positive real number $t$.
By the definition of the Thompson metric in \eqref{E:Tho} and \eqref{E:25}, we obtain that
\begin{equation*}
\|\psi(A)-\psi(B)\|=\frac{1}{\gamma}d_T(A,B), \quad A,B\in \App.
\end{equation*}
The generalized Mazur-Ulam theorem above can be applied for the metric $\frac{1}{\gamma}d_T(.,.)$ on the point-reflection geometry $\App$ and the metric of the $C^*$-norm $\|.\|$ on the point-reflection geometry $\Ams$ (we consider the operations $A\diamond B=AB^{-1}A$ on $\App$ and $A\star B=2A-B$ on $\Ams$.) From Theorem \ref{T:GMU} we obtain that $\psi$ is an isomorphism between point-reflection geometries, namely, it satisfies
\begin{equation}\label{E:30}
\psi(AB^{-1}A)=2\psi(A)-\psi(B), \quad A,B\in \App.
\end{equation}
It is well-known (sometimes it is called Anderson-Trapp theorem) that the geometric mean  
\[
A \sharp B=A^{\fel}(A^{-\fel}B A^{-\fel})^{\fel}A^{\fel}
\]
is the unique solution $X\in \App$ of the equation $XA^{-1}X=B$ for any given $A,B\in \App$. 
Similarly, the arithmetic mean $(A+B)/2$ is the unique solution $X\in \Ams$ of the equation $2X-A=B$ for any given $A,B\in \Ams$. From \eqref{E:30} we can now conclude that 
\begin{equation*}
\psi(A\sharp B)=\frac{\psi(A)+\psi(B)}{2},\quad A,B\in \App.
\end{equation*}
Indeed, choosing $X=A\sharp B$, we have
\[
\psi(B)=\psi(XA^{-1}X)=2\psi(X)-\psi(A)
\]
which implies
\[
\psi(A\sharp B)=\psi(X)=\frac{\psi(A)+\psi(B)}{2},\quad A,B\in \App.
\]
It means that the bijective map $\psi$ transforms the geometric mean on $\App$ to the arithmetic mean on $\Ams$. Proposition 7 in \cite{MolJim} tells that this can happen only when $\A$ is commutative. But in that case, for any $A,B\in \App$ we have  $A^{\gamma}\leq B^{\gamma}$ if and only if $\log A\leq \log B$. Now, applying (c3), one can easily complete the proof referring to the already used fact that Jordan *-isomorphisms preserve commutativity in both directions.
\end{proof}

Beside the order related characterizations given in the first part of the section, we will also need some conditions for positive invertible elements of a $C^*$-algebra, the fulfillment of each of which implies that the elements in question are necessarily central. Our first corresponding result reads as follows.

\begin{lemma}\label{L:2}
Let $\A,\B$ be $C^*$-algebras with faithful traces $\tau$ and $\omega$, respectively. Let $J:\A\to \B$ be a Jordan *-isomorphism and $C\in \Ampp$ be such that $\omega(CJ(X))=\tau(X)$ holds for all $X\in \A$. Then $C$ is necessarily a central element in $\B$.
\end{lemma}

\begin{proof}
We assume that $\omega$ is of norm 1. Then the corresponding GNS construction gives us a *-representation $\pi:\B\to B(H)$ on some Hilbert space $H$ with a cyclic unit vector $\xi \in H$ such that
$\langle \pi(B)\xi,\xi\rangle=\omega(B)$, $B\in \B$. By the faithfulness of $\omega$, the representation $\pi$ is clearly injective, therefore, $\pi$ is an isometry. It follows that $\pi(\B)\subset B(H)$ is a $C^*$-subalgebra. We denote its weak closure by $\mathcal C$. By Proposition 3.19 in \cite{Tak}, there is a faithful normal trace $\nu$ on $\mathcal C$ such that $\nu(\pi(B))=\omega(B)$, $B\in \B$. Let $D=\pi(C)$. Clearly, $D\in \mathcal C$ is positive invertible.
We have
\begin{equation}\label{E:7}
\nu(D\pi(J(X)))=\nu(\pi(CJ(X)))=\omega(CJ(X))=\tau(X), \quad X\in \A.
\end{equation}
Denote $G=\pi\circ J$ which is an (injective) Jordan *-homomorphism from the $C^*$-algebra $\A$ into the von Neumann algebra $\mathcal C$.
By Theorem 3.3 in \cite{Sto}, $G$ is the direct sum of a *-homomorphism and a *-antihomomorphism. Namely, there are orthogonal central projections $P,Q$ in $\mathcal C$ with $P+Q=I$ such that $X\mapsto G(X)P$ is a *-homomorphism and 
$X\mapsto G(X)Q$ is a *-antihomomorphism.
We compute
\begin{equation*}
\begin{gathered}
\nu(DG(XY))=\nu(DG(XY)P+DG(XY)Q)
=\nu(DG(X)G(Y)P+DG(Y)G(X)Q), \quad X,Y\in \A.
\end{gathered}
\end{equation*}
In a similar way, we have
\begin{equation*}
\begin{gathered}
\nu(DG(YX))=
\nu(DG(Y)G(X)P+DG(X)G(Y)Q), \quad X,Y\in \A.
\end{gathered}
\end{equation*}
But, by \eqref{E:7}, $\nu(DG(XY))=\tau(XY)=\tau(YX)=\nu(DG(YX))$, $X,Y\in \A$. Hence, from the last two displayed equalities, we deduce that
\begin{equation*}
\nu(DG(X)G(Y)P-DG(Y)G(X)P)+
\nu(DG(Y)G(X)Q-DG(X)G(Y)Q)=0 , \quad X,Y\in \A.
\end{equation*}
Using the facts that $P,Q$ are central projections in $\mathcal C$ and $\nu$ is a trace, it follows that
\begin{equation*}
\nu(P(DG(X)-G(X)D)G(Y)+Q(G(X)D-DG(X))G(Y))=0 , \quad X,Y\in \A.
\end{equation*}
Since $G(Y)$ runs through the set $\pi(\B)$ which is weak operator dense in $\mathcal C$, and $\nu$ is normal (in fact, it is a vector state as it can be seen in the proof of Proposition 3.19 in \cite{Tak}), using the faithfulness of $\nu$, we conclude that 
\begin{equation*}
P(DG(X)-G(X)D)+Q(G(X)D-DG(X))=0
\end{equation*}
holds for all $X\in \A$. Multiplying this equality by $P$ and $Q$, respectively, from the left, we see that $P(DG(X)-G(X)D)$, $Q(DG(X)-G(X)D)$ are both zero for all $X\in  \A$. This gives that $DG(X)-G(X)D=0$, $X\in \A$. Therefore, $D$ is central in $\pi(\B)$ and, since $\pi$ is injective, we conclude that $C$ is central in $\B$. The proof is complete.
\end{proof}

We will also need the following generalization of the assertion above. 

\begin{lemma}\label{L:3}
Let $\A,\B$ be $C^*$-algebras with faithful traces $\tau$ and $\omega$, respectively. Let $J:\A\to \B$ be a Jordan *-isomorphism, $C\in \Ampp$ and $\beta$ be a positive real number such that $\omega\ler{(CJ(A^{\frac{1}{\beta}})C)^{{\beta}}}=\tau(A)$ holds for all $A\in \App$. Then $C$ is necessarily a central element in $\B$.
\end{lemma}

\begin{proof}
First, we clearly have
$\omega\ler{(CBC)^{\beta}}=\tau\ler{(J^{-1}(B))^\beta}$, $B\in \B$. For any $D\in \Ampp$ and $t\geq 0$, plug $I+tD$ in the place of $B$.
We have
\begin{equation*}
\omega\ler{(C^2+t(CDC))^{\beta}}=\tau\ler{(I+tJ^{-1}(D))^\beta}.
\end{equation*}
Differentiate  with respect to $t$ at $t=0$ and apply Lemma \ref{L:6}. We deduce
\begin{equation*}
\beta\omega\ler{(C^{2(\beta -1)})CDC}=\beta\tau\ler{J^{-1}(D)}.
\end{equation*}
It follows that $\omega(C^{2\beta}D)=\tau(J^{-1}(D))$ holds for all $D\in \Ampp$. Since $\Ampp$ linearly generates $\Am$, we get that the latter equality holds for all elements $D$ of $\B$.
Applying Lemma \ref{L:2}, we apparently obtain that $C^{2\beta}$ and hence also $C$ are central elements in $\B$.
\end{proof}

The last condition concerning centrality that we will use is given in the following assertion.

\begin{lemma}\label{L:5}
Let $\A,\B$ be $C^*$-algebras with faithful traces $\tau$ and $\omega$, respectively. Let $J:\A\to \B$ be a Jordan *-isomorphism and $X_0\in \Ams$ be a self-adjoint element such that 
\begin{equation}\label{E:17}
\omega(\exp(J(T)+X_0))=\tau(\exp(T))
\end{equation}
holds for all $T\in \As$. Then $X_0$ is necessarily a central element in $\B$.
\end{lemma}

\begin{proof}
For any self-adjoint element $S\in \As$ and real number $t$, put $tS$ in the place of $T$ in \eqref{E:17}. Differentiating with respect to $t$ at $t=0$, by Lemma \ref{L:6} we obtain
\begin{equation*}
\tau\ler{(\exp{X_0})J(S)}=\tau(S), \quad S\in \As.
\end{equation*}
Applying Lemma \ref{L:2}, we infer that $\exp(X_0)$ and hence also $X_0$ are central elements in $\B$.
\end{proof}

Observe that by our results Lemmas \ref{L:2}-\ref{L:5}, we have the following complement to Theorem \ref{T:pot}.

\begin{corollary}\label{C:C}
Assume that $\A,\B$ are $C^*$-algebras with faithful traces $\tau$ and $\omega$, respectively. If the surjective positive homogeneous map $\phi:\App \to \Bpp$ in Theorem \ref{T:pot} is also trace-preserving meaning that $\omega (\phi(A))=\tau(A)$, $A\in \App$, then we obtain that the elements $C,X_0$ in (c1) and (c3), respectively, are necessarily central. Therefore, it follows that in that case the transformation $\phi$ is of the form $\phi(A)=DJ(A)$, $A\in \App$, where $D\in \Bpp$ is a central element and $J:\A \to \B$ is a Jordan *-isomorphism.
\end{corollary}

After those long preparations we are now in a position to present the proofs of our main results. We begin with that of Theorem \ref{T:Cone}.

\begin{proof}[Proof of Theorem \ref{T:Cone}]
We deal only with the necessity parts of the statements corresponding to the different quantities \eqref{E:Q1}-\eqref{E:Q5}. The sufficiency parts are easy to check using the properties of Jordan *-isomorphisms listed previously (after Lemma \ref{L:4}).

I. We first assume that $\phi:\App\to \Ampp$ is a surjective map which satisfies \eqref{E:1}. 
Using Lemma \ref{L:1}, we have that $\phi$ has the following order preserving property: for any pair of elements $A,B\in \App$, we have $A^{\frac{\alpha}{z}}\leq B^{\frac{\alpha}{z}}$ if and only if $\phi(A)^{\frac{\alpha}{z}}\leq \phi(B)^{\frac{\alpha}{z}}$. Indeed, we have
\begin{equation}
\begin{gathered}
A^{\frac{\alpha}{z}}\leq B^{\frac{\alpha}{z}} \Leftrightarrow
\tau\hy\Qaz(A||X)\leq \tau\hy\Qaz(B||X), \enskip X\in \App\\ \Leftrightarrow
\omega \hy\Qaz(\phi(A)||\phi(X))\leq \omega \hy\Qaz(\phi(B)||\phi(X)), \enskip X\in \App \Leftrightarrow
\phi(A)^{\frac{\alpha}{z}}\leq \phi(B)^{\frac{\alpha}{z}}.
\end{gathered}
\end{equation} 
Furthermore, from the original preserver property \eqref{E:1} we also easily conclude that $\phi$ is positive homogeneous.
To verify this, for given $A\in\App$ and $t>0$, and for arbitrary $X\in \App$ we compute as follows
\begin{equation}\label{E:22}
\begin{gathered}
\omega\hy\Qaz(\phi(tA)||\phi(X))=
\tau\hy\Qaz(tA||X)=
t^\alpha \tau\hy\Qaz(A||X)\\=
t^\alpha \omega\hy\Qaz(\phi(A)||\phi(X))=
\omega\hy\Qaz(t\phi(A)||\phi(X)).
\end{gathered}
\end{equation}
By Lemma \ref{L:1}, we conclude that $\phi(tA)=t\phi(A)$. If we write $B=A$ in \eqref{E:1}, we get $\omega(\phi(A))=\tau (A)$, $A\in \App$.
We can apply (c1) in Theorem \ref{T:pot} and Corollary \ref{C:C}
to deduce that there are a central element $C\in \Ampp$ and a Jordan *-isomorphism $J:\A\to \B$ such that $\phi$ is of the form $\phi(A)=CJ(A)$, $A\in \App$.
This gives us the necessity parts of the statements concerning the quantities in \eqref{E:Q1}, \eqref{E:Q2} and \eqref{E:Q3}.

II. We next examine the case of the quantity in \eqref{E:Q4}. Let $\phi:\App \to \Ampp$ be a surjective map with the property that
\begin{equation}\label{E:21}
\omega\hy\maxQa(\phi(A)||\phi(B))=\tau\hy\maxQa(A||B), \quad A,B\in \App.
\end{equation}
We can follow the argument given in the previous part of the proof and apply Lemma \ref{L:7} to see that $\phi$ is a positive homogeneous order isomorphism from $\App$ onto $\Ampp$. Putting $B=A$ into the equality \eqref{E:21}, we get $\omega(\phi(A))=\tau (A)$, $A\in \App$. Therefore, applying (c1) in Theorem \ref{T:pot} and Corollary \ref{C:C} again, we obtain the required conclusion. 

III. Finally, we turn to the case of the quantity in \eqref{E:Q5}.
Let $\phi:\App\to \Ampp$ be a surjective map which satisfies
\begin{equation}\label{E:4}
\omega\hy\moQa(\phi(A)||\phi(B))=\tau\hy\moQa(A||B), \quad A,B\in \App.
\end{equation}
By Lemma \ref{L:4} we deduce that $\phi$ has the following property: for any $A,B\in \App$ we have $\log A\leq \log B$ if and only if $\log \phi(A)\leq \log \phi(B)$. Applying a simple calculation of the style of \eqref{E:22} and referring to Lemma \ref{L:4} again, we obtain that $\phi$ is positive homogeneous. 
If we put $B=A$ into \eqref{E:4}, we get $\omega(\phi(A))=\tau (A)$, $A\in \App$. By (c3) in Theorem \ref{T:pot} and Corollary \ref{C:C}, we can trivially complete the necessity part of the proof concerning the quantity \eqref{E:Q5}.

As already mentioned, the sufficiency parts of the corresponding statements are easy to check by applying the previously listed properties of Jordan *-isomorphisms. This finishes the proof of the theorem.
\end{proof}

We proceed with the following comment.
When not restricted to the density space, but defined and studied 
on the whole positive definite cone of a matrix algebra (i.e., the full operator algebra over a finite dimensional Hilbert space), the quantities in \eqref{E:Q1}-\eqref{E:Q5} are in fact usually normalized by the trace of the first variable (see, for example, Definitions 4.3 and 4.5 in \cite{Toma}). Apparently, modifying the problems what we have solved in Theorem \ref{T:Cone} in that way, we can arrive at a new collection of problems. We do not want to deal with all those questions in details since here our focus is on maps defined on density spaces (which are the most relevant problems we believe) and, as we will see soon, the required results can be derived directly from Theorem \ref{T:Cone}. However, we pick one such question, the one related to the quantity $\moQa$ and demonstrate that, with some modifications, our approach developed above can be adopted to that setting, too. We believe that with more or less difficulties all other problems concerning the normalized versions of the quantities in \eqref{E:Q1}-\eqref{E:Q5} could be solved as well. The reason for choosing exactly the quantity $\moQa$ is that the corresponding symmetry transformations have been determined recently in the paper \cite{NagyGaalLMP} in the finite dimensional case. Hence, the following result is a far reaching generalization of Theorem 1 in \cite{NagyGaalLMP} for the case of abstract $C^*$-algebras.

\begin{theorem}\label{T:Conedivided}
Let $\A, \B$ be $C^*$-algebras with faithful traces $\tau,\omega$, respectively, and let $\alpha$ be a given positive real number different from 1. Assume that $\phi:\App\to \Ampp$ is a surjective map. Then $\phi$ satisfies
\begin{equation}\label{E:11}
\frac{\omega\hy\moQa(\phi(A)||\phi(B))}{\omega(\phi(A))}=\frac{\tau\hy\moQa(A||B)}{\tau (A)}, \quad A,B\in \App
\end{equation} 
if and only if there are a central element $C\in \Ampp$ and a Jordan *-isomorphism $\phi:\A\to \B$ such that
$\phi(A)=CJ(A)$ holds for all $A\in \App$ and
$\omega(CJ(X))=\frac{\omega(C)}{\tau(I)}\tau(X)$ is satisfied for all $X\in \A$.

In particular, any surjective map $\phi:\App\to \Ampp$ which satisfies \eqref{E:11} is necessarily a constant multiple of a map of the standard form.
\end{theorem}

\begin{proof}
Assume that the surjective map $\phi:\App\to \Ampp$ satisfies
\eqref{E:11}. By the first part of the statement in Lemma \ref{L:4}, we have the following equivalence: for any $B,B'\in \App$, the inequality $(1-\alpha)\log B\leq (1-\alpha)\log B'$ holds if and only if 
\begin{equation*}\label{E:50}
\frac{\tau\hy\moQa(X||B)}{\tau (X)} \leq \frac{\tau\hy\moQa(X||B')}{\tau (X)}, \quad X\in \App.
\end{equation*}
One can easily deduce from this characterization that $\phi$ has the property that $\log B\leq \log B'$ if and only if $\log \phi(B)\leq \log \phi(B')$ for any $B,B'\in \App$ and then that $\phi$ is positive homogeneous. By (c3) in Theorem \ref{T:pot} we infer that
there is a Jordan *-isomorphism $J:\A \to \B$ and an element $X_0\in \Ams$ such that $\phi(A)=\exp(J(\log A)+X_0)$, $A\in \App$. We claim that $X_0$ is a central element in $\Am$.
In fact, using \eqref{E:11}, we have
\begin{equation*}
\frac{\omega\ler{\exp(\alpha J(\log A)+(1-\alpha)J(\log B)+X_0)}}{\omega\ler{\exp( J(\log A)+X_0)}}=\frac{\tau\ler{\exp(\alpha \log A +(1-\alpha)\log B)}}{\tau\ler{\exp(\log A)}}, \quad A,B\in \App.
\end{equation*}
Then, with $\beta=1/\alpha$, we can rewrite this as
\begin{equation*}
\frac{\omega(\exp(J(T)+J(S)+X_0))}{\omega(\exp(\beta J(T)+X_0))}=\frac{\tau(\exp(T+S))}{\tau(\exp(\beta T))}, \quad T,S\in \As.
\end{equation*}
Since the elements $T+S$ and $\beta T$ are in fact independent, we infer that
\begin{equation*}\label{E:24}
\frac{\omega(\exp(J(R)+X_0))}{\omega(\exp(J(R')+X_0))}=\frac{\tau(\exp(R))}{\tau(\exp(R'))}, \quad R,R'\in \As.
\end{equation*}
This implies that
\begin{equation}\label{E:24a}
\omega(\exp(J(R)+X_0))=\delta \tau(\exp(R)), \quad R\in \As
\end{equation}
holds with some positive constant $\delta$. 
Writing $\delta=\exp{\gamma}$ for some $\gamma\in \R$, we have 
\begin{equation*}\label{E:24b}
\omega(\exp(J(R)+(X_0-\gamma I)))=\tau(\exp(R)), \quad R\in \As.
\end{equation*}
Using Lemma \ref{L:5}, we deduce that $X_0-\gamma I$ and hence also $X_0$ are central elements of $\B$. Therefore, with the central element $D=\exp(X_0)$ in $\B$, it follows that 
\[
\phi(A)=\exp(J(\log A)+X_0)=DJ(A), \quad A\in \B.
\]
Plugging $R=\log A$ into \eqref{E:24a}, we have $\omega(DJ(A))=\delta \tau(A)$, $A\in \App$. Choosing $A=I$, it is now obvious that $\delta=\frac{\omega(D)}{\tau(I)}$. We obtain
\[
\omega(DJ(A))=\frac{\omega(D)}{\tau(I)} \tau(A), \quad A\in \App.
\]
By linearity, the above equality holds also on the whole algebra $\A$.
This completes the proof of the necessity part of the theorem.

The sufficiency part can easily be checked.
\end{proof}

We next present the proof of our statement concerning quantum R\'enyi relative entropy preservers between density spaces of $C^*$-algebras which is one of our main goals in this paper.

\begin{proof}[Proof of Corollary \ref{C:main}]
As above, we check only the necessity parts of the statement, the sufficiency follows by easy computations.
Our strategy is simple and applies in the case of any of the considered quantum R\'enyi relative entropies. In fact, the only thing we have to do is the following. We extend the given transformation $\phi:\DtAi \to \DtAmi$ in the statement to a surjective map $\psi:\App\to \Ampp$ between positive definite cones by the simple formula 
\begin{equation}\label{E:ext}
\psi(A)=\tau(A)\phi \ler{\frac{A}{\tau(A)}}, \quad A\in \App.
\end{equation}
It can easily be checked that this extension $\psi$ satisfies the conditions in Theorem \ref{T:Cone}. Application of that result apparently gives us the desired formula for $\phi$.
\end{proof}

We now turn to the proof of our theorem about the essential difference among the considered quantum R\'enyi relative entropies.

\begin{proof}[Proof of Theorem \ref{T:commut}]
I. In the first part of the proof let $\phi:\DtAi\to \DtAmi$ be a surjective map which satisfies
\begin{equation*}\label{E:51}
\omega\hy\moDa((\phi(A)||\phi(B))=\tau\hy\Daz(A||B), \quad A,B\in \DtAi.
\end{equation*} 
By the formula given in  
\eqref{E:ext} we extend the transformation $\phi:\DtAi\to \DtAmi$ to a map (denoted by the same symbol) $\phi:\App\to \Ampp$. This new transformation is a surjective map between positive definite cones and it can easily be verified that it satisfies 
\begin{equation*}\label{E:14}
\omega\hy\moQa(\phi(A)||\phi(B))=\tau\hy\Qaz(A||B), \quad A,B\in \App.
\end{equation*}
By Lemmas \ref{L:1} and \ref{L:4} we have that for any $A,B\in \App$, the inequality $A^{\frac{\alpha}{z}}\leq B^{\frac{\alpha}{z}}$ holds if and only if $\log \phi(A)\leq \log \phi(B)$ is valid. By its construction $\phi$ is obviously positive homogeneous. Applying (c4) in Theorem \ref{T:pot} we infer that $\A,\B$ are commutative and we are done.

As for other pairs of quantum R\'enyi relatives entropies,
we can continue in a similar fashion. 

II. Suppose that $\phi:\DtAi\to \DtAmi$ is a surjective map such that 
\begin{equation}\label{E:44}
\omega\hy\moQa(\phi(A)||\phi(B))=\tau\hy\maxQa(A||B)
\end{equation}
for all $A,B\in \DtAi$. Here we assume that $\alpha\leq 2$.
Applying the method of extensions \eqref{E:ext} we can extend $\phi$ to a positive homogeneous surjective map denoted by the same symbol $\phi$ which is defined between the positive definite cones $\App,\Ampp$ and satisfies \eqref{E:44} for all $A,B\in \App$.
By Lemmas \ref{L:7} and \ref{L:4} we obtain that for any $A,B\in \App$, the inequality $A\leq B$ holds if and only if $\log \phi(A)\leq \log \phi(B)$. (We remark that here we need to consider the cases $\alpha<1$ and $1<\alpha \leq 2$ separately, see Lemma \ref{L:7}.) One can argue and complete the proof in the same way as in part I of the proof. 

III. In the third part of the proof
we assume that, after employing the extension method \eqref{E:ext}, $\phi:\App\to \Ampp$ is a positive homogeneous surjective map such that 
\begin{equation*}
\omega\hy\Qazv(\phi(A)||\phi(B))=\tau\hy\Qaz(A||B), \quad A,B\in \App
\end{equation*}
holds for some positive real numbers $z,z'$ with $z\neq z'$. Using Lemma \ref{L:1}, we obtain that for any $A,B\in \App$, we have that $A^\frac{\alpha}{z}\leq B^\frac{\alpha}{z}$ holds if and only if $\phi(A)^\frac{\alpha}{z'}\leq \phi(B)^\frac{\alpha}{z'}$. Applying (c2) in Theorem \ref{T:pot} we conclude that $\A,\B$ are necessarily commutative.

IV. In the last part, again after applying the method of extension \eqref{E:ext},  we assume that $\phi:\App \to \Bpp$ is a positive homogeneous surjective map such that 
\begin{equation}\label{E:60}
\omega\hy\maxQa(\phi(A)||\phi(B))=\tau\hy\Qaz(A||B)
\end{equation}
holds for all $A,B\in \Bpp$. 

We have to distinguish two cases. First, assume that $1<\alpha \leq 2$, this is the more complicated case. By Lemma \ref{L:1} and Lemma \ref{L:7} we deduce that for any $A,B\in \App$, $A^\frac{\alpha-1}{z}\leq  
B^\frac{\alpha-1}{z}$ if and only if $\phi(A)\leq \phi(B)$. Assuming $\alpha-1\neq z$, by (c2) in Theorem \ref{T:pot} we obtain that the algebras $\A,\B$ are commutative. If $\alpha-1=z$, then (c1) in Theorem \ref{T:pot} applies. Since, by its construction, $\phi$ is trace-preserving, using Corollary \ref{C:C} as well, we have
a central element $C\in \Bpp$ and a Jordan *-isomorphism $J:\A \to \B$ such that
$\phi(A)=CJ(A)$, $A\in \App$. Applying \eqref{E:60} for arbitrary $A,B\in \App$ and using the identity $z=\alpha -1$, we have
\begin{equation*}
\omega \ler{C J(B)^{\fel} \ler{J(B)^{-\fel}J(A)J(B)^{-\fel}}^\alpha J(B)^{\fel}}=\tau\ler{B^{-\fel}A^{\frac{\alpha}{\alpha-1}}B^{-\fel}}^{\alpha-1}, \quad A,B\in \App.
\end{equation*}
By the properties of Jordan *-isomorphisms we get
\begin{equation}\label{E:P2}
J(B)^{\fel} \ler{J(B)^{-\fel}J(A)J(B)^{-\fel}}^\alpha J(B)^{\fel}=J\ler{B^{\fel} \ler{B^{-\fel}A B^{-\fel}}^\alpha B^{\fel}}, \quad A,B\in \App.
\end{equation}
Using the trace-preserving property
\begin{equation*}\label{E:63}
\omega (CJ(A))=\tau (A), \quad A\in \App,
\end{equation*}
we deduce that
\begin{equation}\label{E:64}
\tau \ler{ B^{\fel} \ler{B^{-\fel}AB^{-\fel}}^\alpha B^{\fel}}=\tau\ler{B^{-\fel}A^{\frac{\alpha}{\alpha-1}}B^{-\fel}}^{\alpha-1}, \quad A,B\in \App.
\end{equation}
We claim that this implies that $\alpha =2$. To show that, first observe that by Lemma \ref{L:20} we have
\begin{equation}\label{E:65}
B^{\fel} \ler{B^{-\fel}AB^{-\fel}}^\alpha B^{\fel}=
A^{\fel} \ler{A^{\fel}B^{-1}A^{\fel}}^{\alpha-1} A^{\fel}, \quad A,B\in \App.
\end{equation}
As we have referred to that in the proof of Lemma \ref{L:1}, for every $X,Y\in \App$ we have that $XY^2X$ and $YX^2Y$ are unitarily equivalent. Therefore, we compute
\begin{equation}\label{E:66}
\tau\ler{B^{-\fel}A^{\frac{\alpha}{\alpha-1}}B^{-\fel}}^{\alpha-1}=\tau \ler{A^\frac{\alpha}{2(\alpha-1)}B^{-1}A^\frac{\alpha}{2(\alpha-1)}}^{\alpha-1}  , \quad A,B\in \App.
\end{equation}
Using \eqref{E:64}, \eqref{E:65} and \eqref{E:66}, it follows that
\begin{equation*}
\tau\ler{A^{\fel} \ler{A^{\fel}B^{-1}A^{\fel}}^{\alpha-1} A^{\fel} }= \tau \ler{A^\frac{\alpha}{2(\alpha-1)}B^{-1}A^\frac{\alpha}{2(\alpha-1)}}^{\alpha-1}  , \quad A,B\in \App.
\end{equation*}
Let $X,Y\in \App$ be arbitrary and choose $A,B\in \App$ such that $A^{\fel}B^{-1}A^{\fel}=X$, $A^{\fel}=Y^{\alpha-1}$.
Then, from the above displayed formula we can derive
\begin{equation*}
\tau(Y^{\alpha-1}X^{\alpha-1}Y^{\alpha-1})=\tau(YXY)^{\alpha-1}, \quad X,Y\in \App.
\end{equation*}
From this identity, using the characterization of the order given in Lemma \ref{L:1}, we obtain that for any $X,X'\in \App$, the relation $X\leq X'$ holds if and only if $X^{\alpha-1} \leq X'^{\alpha-1}$ is valid. 
Assume that $\A$ is non-commutative.
Since the exponent $\alpha-1$ is positive, we obtain (referring to Theorem 2 in \cite{NUW}) that $\alpha-1=1$, that is $\alpha=2$. 
But we then clearly have $\tau\hy D_2^{max}(.||.)=\tau\hy D_{2,1}(.||.)=\tau\hy D_2^c(.||.)$ meaning that the quantum R\'enyi divergences what we are considering are not different, a contradiction. Therefore, $\A$ is necessarily commutative and  because $J$ is a Jordan *-isomorphism between $\A$ and $\B$, hence $\B$ is commutative, too.
This completes the proof in the case where $1<\alpha \leq 2$.

Let us finally examine the case where $\alpha<1$.
Similarly to what we have done above,  by Lemma \ref{L:1} and Lemma \ref{L:7} we deduce that for any $A,B\in \App$, the inequality $A^{\frac{\alpha}{z}}\leq B^{\frac{\alpha}{z}}$ holds if and only if $\phi(A)\leq \phi(B)$ is valid. 
If $\alpha\neq z$, then by (c4) in Theorem \ref{T:pot} we have that $\A,\B$ are commutative. In the case where $\alpha =z$, referring to the fact that $\phi$ is trace-preserving (because of its construction), we obtain by (c1) in Theorem \ref{T:pot} and Corollary \ref{C:C} that there are a central element $C\in \Bpp$ and a Jordan *-isomorphism $J:\A \to \B$ such that
$\phi(A)=CJ(A)$, $A\in \App$.
Using  \eqref{E:60} and the identity $z=\alpha$, we compute 
\begin{equation*}
\omega \ler{C J(B)^{\fel} \ler{J(B)^{-\fel}J(A)J(B)^{-\fel}}^\alpha J(B)^{\fel}}=\tau\ler{B^{\frac{1-\alpha}{2\alpha}}AB^{\frac{1-\alpha}{2\alpha}}}^{\alpha}, \quad A,B\in \App.
\end{equation*}
Applying \eqref{E:P2} and the trace-preserving property of $\phi$, it follows that
\begin{equation}\label{E:69}
\begin{gathered}
\tau \ler{ B^{\fel}(B^{-\fel}A B^{-\fel})^\alpha B^{\fel}}=
\tau\ler{B^{\frac{1-\alpha}{2\alpha}}AB^{\frac{1-\alpha}{2\alpha}}}^{\alpha}, \quad A,B\in \App.
\end{gathered}
\end{equation}
Let now $T,S\in \App$ be arbitrary. We can choose $A,B\in \App$ such that
$B^{-\fel}A B^{-\fel}=T$ and $B^\fel=S^\alpha$. Using \eqref{E:69} we can derive
\begin{equation*}
\tau (S^\alpha T^\alpha S^\alpha)=\tau (STS)^\alpha, \quad S,T\in \App.
\end{equation*}
Assume that $\A$ is non-commutative.
Arguing just as in the former part of the proof concerning the case $\alpha>1$, we would conclude that $\alpha$ is necessarily 1, a contradiction. Therefore, $\A$ and then $\B$, too, are commutative.

The proof of the theorem is now complete.
\end{proof} 

In the last part of the paper we present the proofs of our results concerning the Umegaki and the Belavkin-Staszewski  relative entropies. Similarly as above, our arguments rest on characterizations of the order in terms of the relative entropies in question. In what follows we consider the  quantities $S_U^{\tau}(.||.)$ and $S_{BS}^{\tau}(.||.)$ on the whole positive definite cone defined by the same formula \eqref{E:U} and \eqref{E:BS}, respectively.

\begin{lemma}\label{L:U}
Let $\A$ be a $C^*$-algebra with a faithful trace $\tau$. Select $A,B\in \App$. We have $\log A\leq \log B$ if and only if $S_U^{\tau}(X||B)\leq S_U^{\tau}(X||A)$ holds for all $X\in \App$.
\end{lemma}

\begin{proof}
Clearly, we have that $S_U^{\tau}(X||B)\leq S_U^{\tau}(X||A)$ holds for all $X\in \App$ if and only if 
\[
\tau(X\log A)\leq \tau(X\log B),\quad X\in \App.
\]
This is equivalent to $\log A\leq \log B$, cf. the last part of the proof of Lemma \ref{L:4}.
\end{proof}

We can now present the proof of Theorem \ref{T:Udens}. As we have mentioned in the first section of the paper, in \cite{ML17a} we described the structure of all bijective maps between the positive definite cones of $C^*$-algebras with faithful traces which are invariance transformations under $S_U^{\tau}(.||.)$. To be honest, in Theorem 1 in that paper we assumed that the transformations had the same domain and codomain and that the trace was normalized, assigned 1 to the identity. 
However, one can easily see that those restrictions in \cite{ML17a} are not crucial, and we could apply an appropriately  modified version of the result there to prove Theorem \ref{T:Udens}. 
Let us also mention that the approach in \cite{ML17a} was completely different from what we follow here, not relied on structural theorems of Thompson isometries and related maps. 

\begin{proof}[Proof of Theorem \ref{T:Udens}]
One could argue as follows. Let $\phi:\DtAi\to \DtAmi$ be a surjective map such that
\begin{equation*}
S_U^{\omega}(\phi(A)||\phi(B))=S_U^{\tau}(A||B), \quad A,B\in \DtAi.
\end{equation*} 
Applying the extension formula given in \eqref{E:ext}, one can check that $\phi$ extends to a surjective map from $\App$ onto $\Ampp$, denoted by the same symbol $\phi$,  satisfying
\begin{equation}\label{E:55}
S_U^{\omega}(\phi(A)||\phi(B))=S_U^{\tau}(A||B), \quad A,B\in \App.
\end{equation}
By Lemma \ref{L:U} we have that for any $A,B\in \App$, the inequality $\log \phi(A)\leq \log \phi(B)$ holds if and only if $\log A\leq \log B$ is valid. 
This implies that that $\phi$ is injective and hence bijective. Referring to the remark before the present proof, we could now use Theorem 1 in \cite{ML17a} and conclude  that there are a central element $C\in \Ampp$ and a Jordan *-isomorphism $J:\A\to \Am$ such that $\phi(A)=CJ(A)$, $A\in \App$ and $\omega(CJ(A)J(B))=\tau(AB)$ holds for all $A,B\in\App$. Since, $J(I)=I$, we could finish the proof of the necessity part of the theorem. 

However, we can give also a direct argument following the general approach of the present paper. Indeed, we have that for any $A,B\in \App$, the inequality $\log A\leq \log B$ holds if and only if $\log \phi(A)\leq \log \phi(B)$. Moreover, by its construction, the extension $\phi$ is clearly positive homogeneous and trace-preserving. We apply (c3) in Theorem \ref{T:pot} and Corollary \ref{C:C} to conclude that
\[
\phi(A)=CJ(A), \quad A\in \App
\]
holds with some central element $C\in \Bpp$ and Jordan *-isomorphism $J:\A \to \Am$ and we obtain the necessity part of the statement.

As for the sufficiency, it requires only a little bit of nontrivial calculation. Assume that $\phi$ is of the form $\phi(A)=CJ(A)$, $A\in \DtAi$ where $C\in \Ampp$ is a central element, $J:\A\to \Am$ is a Jordan *-isomorphism and $\omega(CJ(X))=\tau(X)$ holds for all $X\in \A$. Clearly, $\phi:\DtAi\to \DtAmi$ is a bijective map. Moreover, we compute
\begin{equation*}
\begin{gathered}
2\omega (CJ(A)J(B))=\omega\ler{C(J(A)J(B)+J(B)J(A))}=
\omega\ler{CJ(AB+BA)}\\=\tau(AB+BA)=2\tau(AB), \quad A,B\in \A.
\end{gathered}
\end{equation*}
Using this equality and the properties of Jordan *-isomorphisms, we easily conclude that \eqref{E:55} is satisfied: 
\begin{equation*}
\begin{gathered}
\omega \ler{CJ(A)(\log CJ(A)-\log CJ(B))}=
\omega \ler{CJ(A)(\log J(A)-\log J(B))}\\=
\omega \ler{CJ(A)(J(\log A)-J(\log B))}=
\tau \ler{A(\log A-\log B)}, \quad A,B\in \DtAi.
\end{gathered}
\end{equation*}
This completes the proof of the theorem.
\end{proof}

We next present the proof of our result concerning the Belavkin-Staszewski relative entropy. Here we again follow our general idea. To do that, we will need the following characterization of the order in terms of the Belavkin-Staszewski relative entropy. The next lemma is an apparent consequence of Lemma \ref{L:BS}. 

\begin{lemma} \label{L:CBS}
Let $\A$ be a $C^*$-algebra with a faithful trace $\tau$. For any $A,B\in \App$, we have
$A\leq B$ if and only if $S^\tau_{BS}(X||B)\leq S^\tau_{BS}(X||A)$ holds for all $X\in \App$. 
\end{lemma}

Now, the proof of Theorem \ref{T:BSdens} is as follows.

\begin{proof}[Proof of Theorem \ref{T:BSdens}]
Let $\phi$ be a surjective map which respects the Belavkin-Staszewski relative entropy, i.e., satisfies \eqref{E:33} on $\DtAi$. Again, by \eqref{E:ext} we extend $\phi$ to a trace-preserving positive homogeneous surjective map $\phi: \App \to \Ampp$ denoted by the same symbol $\phi$. It is easy to verify that $\phi$ satisfies
\begin{equation*}
S^{\omega}_{BS}(\phi(A)||\phi(B))=S^\tau_{BS}(A||B), \quad A,B\in \App.
\end{equation*}
By Lemma \ref{L:CBS} we have that $\phi$ is an order isomorphism between $\App$ and $\Ampp$. Therefore, by (c1) in Theorem \ref{T:pot} and Corollary \ref{C:C}, we obtain that there are a central element $C\in \Ampp$ and a Jordan *-isomorphism $J:\A \to \Am$ such that $\phi(A)=CJ(A)$ holds for all $A\in \App$. This finishes the proof of the necessity part of the theorem. The sufficiency part requires only easy computation.
\end{proof} 

It has remained to verify Theorem \ref{T:UBS}.

\begin{proof}[Proof of Theorem \ref{T:UBS}]
In view of the previous arguments, here we only give the sketch of the proof.

Namely, applying the extension formula \eqref{E:ext}, we extend $\phi$ to a surjective map (denoted by the same symbol) $\phi: \App \to \Ampp$ which satisfies
\begin{equation*}\label{E:57}
\omega\ler{\phi(A)(\log \phi(A)-\log \phi(B))}=\tau\ler{A \log(A^{1/2}B^{-1}A^{1/2})}, \quad A,B\in \App.
\end{equation*}
By Lemmas \ref{L:U} and \ref{L:CBS}, we obtain that for any $A,B\in \App$ we have $A\leq B$ if and only $\log \phi(A)\leq \log \phi(B)$. Since, by its construction, $\phi$ is also positive homogeneous, by (c4) in Theorem \ref{T:pot} we conclude that $\A, \B$ are necessarily commutative. 
\end{proof}

\section{Acknowledgement}
The author is very grateful to the referee for his/her kind comments which helped to improve the presentation of the paper. 

\bibliographystyle{amsplain}

\end{document}